\documentclass[mathpazo]{cicp}

\newtheorem{algorithm}{Algorithm}

\begin{document}
\title{An iterative two-grid method for strongly nonlinear elliptic boundary value problems}


 \author[Zhan J J et.~al.]{Jiajun Zhan\affil{1},
       Lei Yang\affil{1}, Xiaoqing Xing\affil{2}, and Liuqiang Zhong\affil{2}\comma\corrauth}
 \address{\affilnum{1}\ School of Computer Science and Engineering, Faculty of Innovation Engineering,
          Macau University of Science and Technology,
          Macao SAR 999078, China. \\
           \affilnum{2}\ School of Mathematical Sciences,
           South China Normal University, Guangzhou 510631, China. }
 \emails{{\tt 2109853gii30011@student.must.edu.mo} (J.~Zhan), {\tt leiyang@must.edu.mo} (L.~Yang),
          {\tt xingxq@scnu.edu.cn} (X.~Xing), {\tt zhong@scnu.edu.cn} (L.~Zhong)}

\begin{abstract}
We design and analyze an iterative two-grid algorithm for the finite element discretizations of strongly nonlinear elliptic boundary value problems in this paper. 
We propose an iterative two-grid algorithm, in which a nonlinear problem is first solved on the coarse space, and then a symmetric positive definite problem is solved on the fine space.
The main contribution in this paper is to  establish  a first convergence analysis, which requires dealing with four coupled error estimates,  for the iterative two-grid methods.
We also present some numerical experiments to confirm the efficiency of the proposed algorithm.
\end{abstract}


\maketitle

\section{Introduction}

The two-grid methods are first proposed for nonselfadjoint problems and indefinite elliptic problems \cite{XuJC92Meeting, XuJCCaiXC92:311}. 
Then, the two-grid methods  are extended to solve  semilinear elliptic problems \cite{XuJC94:231}, quasi-linear and  nonlinear elliptic problems \cite{XuJC94:79, XuJC96:1759}, respectively.
Especially, for nonlinear elliptic problems, the basic idea of two-grid methods is to first obtain a rough solution by solving the original problem in a ``coarse mesh'' with mesh size $H$, and then correct the rough solution by solving a symmetric positive definite (SPD) system in a ``fine mesh'' with mesh size $h$. 
Noticing the ``coarse mesh'' could be extremely coarse in contrast to the ``fine mesh'', it is not difficult to solve an original problem in ``coarse mesh''.
Therefore, two-grid methods reduce the computational complexity of solving the original problem to solving a SPD problem and dramatically improve the computational speed.
Recently, Bi, Wang and Lin \cite{BiCJWangC18:23} presented a two-grid algorithm to solve the strongly nonlinear elliptic problems and provided a posteriori error estimator for the two-grid methods. 
It's noted that the literature mentioned above is all about non-iterative two-grid methods.

As is well-known, the mesh size $H$ of ``coarse mesh'' and $h$ of ``fine mesh'' should satisfy a certain relationship for the optimal convergence order in non-iterative two-grid methods.
The iterative two-grid methods have the advantage over the non-iterative two-grid methods in that, the distance between the mesh sizes $H$ and $h$ can be enlarged by increasing the iteration counts with the same accuracy.
However, there is only a small amount of literature on iterative two-grid methods of conforming finite element discretization for elliptic problems. 
Xu \cite{XuJC96:1759} first proposed and analyzed an iterative two-grid method for non-symmetric positive definite elliptic problems. 
Zhang, Fan and Zhong \cite{ZhangWFFanRH20:522} designed some iterative two-grid algorithms for semilinear elliptic problems and provided the corresponding convergence analysis. 
To our knowledge, there is not any published literature on the iterative two-grid algorithm of conforming finite element discretization for strongly nonlinear elliptic boundary value problems.

In this paper, an iterative two-grid algorithm for solving strongly nonlinear elliptic problems is studied.
The discrete system of strongly nonlinear elliptic problems is presented at first. 
And then, an iterative two-grid algorithm is proposed for the discrete system, which is obtained by applying a non-iterative two-grid algorithm of \cite{XuJC94:79} in a successive fashion.
Finally, a challenging convergence analysis of the proposed algorithm is provided.
Despite the fact that our algorithm is simply obtained by \cite{XuJC94:79}, the convergence analysis of the non-iterative two-grid algorithm could not be directly applied to the iterative two-grid algorithm. 
Here we complete this challenging convergence analysis by mathematical induction which can also be used in solving semilinear elliptic problems by iterative two-grid algorithms in \cite{ZhangWFFanRH20:522}. 
However, we must emphasize that the convergence analysis of our algorithm is significantly different from the one of \cite{ZhangWFFanRH20:522}. 
Compared with the current work \cite{ZhangWFFanRH20:522}, our convergence analysis is far more difficult and complex, and specific challenges could be reflected in:
(1) the higher order derivative component of our model problem is still nonlinear;
(2) the coupled error estimates cause formidable obstacles for the convergence analysis (See the proof of Lemma 4.7).

To avoid the repeated use of generic but unspecified constants, $x \lesssim y $ is used to denote $x \leq C y$, where $C$ are some positive constants which do not depend on the mesh size. 
Furthermore the constants $C$ may denote different values under different circumstances. 
For some specific constants, we use the constant $C$ with some subscript to denote.

The rest of the paper is organized as follows.
In Section \ref{Sec:model}, the discrete scheme of strongly nonlinear elliptic problems, as well as the corresponding well-posedness and priori error estimates,  are introduced.
In Section \ref{Sec:algorithm}, an  iterative two-grid algorithms is proposed. 
And then some preliminaries and the convergence analysis of the proposed algorithms are provided in Section \ref{Sec:convergence}. 
Finally, some numerical experiments are presented to verify the efficiency of the proposed algorithm in Section \ref{Sec:numerical}.

\section{Model problems and discrete systems}\label{Sec:model}

In this section, we present the continuous and discrete variational problems of strongly nonlinear elliptic problems, and provide the corresponding well-posedness and priori error estimates.

Given a bounded convex polygonal domain $\Omega \subset \mathbb{R}^2$ with the boundary $\partial \Omega$. 
We denote $W^{m, p}(\Omega)$ as the standard Sobolev space with norm $\| \cdot \|_{m, p, \Omega}$ and seminorm $\mid \cdot \mid_{m, p, \Omega}$, where the integers $m \geq 0$ and $p \geq 1$.
For convenience, we also denote $H^m(\Omega) = W^{m, 2}(\Omega)$, $\| \cdot \|_{m} = \| \cdot \|_{m, 2, \Omega}$ and $H^1_0(\Omega) := \{u\in H^1(\Omega) : u\mid_{\partial \Omega} =0 \}$.

We consider the following strongly nonlinear elliptic problems:
\begin{equation}\label{Eqn:u}
\left\{\begin{aligned}
- \nabla \cdot \boldsymbol{a}(\boldsymbol{x}, u, \nabla u)+f(\boldsymbol{x}, u, \nabla u) &=0, & & \text { in } \Omega, \\
u &=0, & & \text { on } \partial \Omega, 
\end{aligned}\right.
\end{equation}
where $\boldsymbol{a}(\boldsymbol{x}, y, \boldsymbol{z}) : \bar{\Omega} \times \mathbb{R} \times \mathbb{R}^2 \rightarrow \mathbb{R}^2$ and $f(\boldsymbol{x}, y, \boldsymbol{z}): \bar{\Omega} \times \mathbb{R} \times \mathbb{R}^2 \rightarrow \mathbb{R}$. 
When $\boldsymbol{a}(\boldsymbol{x}, u, \nabla u)$ and $f(\boldsymbol{x}, u, \nabla u)$ take different functions, different problems are available, such as mean curvature flow, Bratu's problem and so on(See \cite{GudiNataraj08:233}). 


We assume that $\boldsymbol{a}(\boldsymbol{x}, y, \boldsymbol{z})$ and $f(\boldsymbol{x}, y, \boldsymbol{z})$ are second order continuous differentiable functions. 
For simplicity, we denote that 
$\boldsymbol{a}_y(w) = D_y\boldsymbol{a}(\boldsymbol{x}, w, \nabla w),  \boldsymbol{a}_{\boldsymbol{z}}(w) = D_{\boldsymbol{z}}\boldsymbol{a}(\boldsymbol{x}, w, \nabla w) $, 
$f_y(w) = D_yf(\boldsymbol{x}, w, \nabla w)$ and $f_{\boldsymbol{z}}(w) = D_{\boldsymbol{z}}f(\boldsymbol{x}, w, \nabla w) $,
and similar notations are applied to the second order derivatives of $\boldsymbol{a}(\boldsymbol{x}, y, \boldsymbol{z})$ and $f(\boldsymbol{x}, y, \boldsymbol{z})$. 

\begin{remark}\label{Rem:Assaf}
Since $\boldsymbol{a}(\boldsymbol{x}, y, \boldsymbol{z})$ and $f(\boldsymbol{x}, y, \boldsymbol{z})$ are second order continuous differentiable functions, 
there exists a positive constant $\tilde{C}$ as upper bound with respect to all the first and second order derivatives of $\boldsymbol{a}(\cdot, \cdot, \cdot)$ and $f(\cdot, \cdot, \cdot)$. 
\end{remark}

We denote 
\begin{equation}\label{Eqn:A}
A(v, \varphi)=(\boldsymbol{a}(\boldsymbol{x}, v, \nabla v), \nabla \varphi)+(f(\boldsymbol{x}, v, \nabla v), \varphi), \quad \forall v, \varphi \in H_0^1(\Omega). 
\end{equation}
By Green formula, it's easy to see that the solution $u \in H_0^1(\Omega)$ of \eqref{Eqn:u} satisfies
\begin{equation}\label{Eqn:Varu}
A(u, v)=0, \quad \forall v \in H_0^1(\Omega). 
\end{equation}


The Fr\'{e}chet derivative $\mathcal{L}^{\prime}$ of \eqref{Eqn:u} at $w $ is given by
$$
\mathcal{L}^{\prime}(w) v=- \nabla \cdot (\boldsymbol{a}_y(w)v + \boldsymbol{a}_{\boldsymbol{z}}(w)\nabla v)+f_y(w)v + f_{\boldsymbol{z}}(w)\nabla v .
$$

In the following, we give some of our basic assumptions (Similar assumptions also could be found in \cite{XuJC96:1759} or \cite{GudiNataraj08:233}).
Firstly, the problem \eqref{Eqn:Varu} has a solution $u \in H_0^1(\Omega) \cap H^{r+1}(\Omega) \cap W^{2, 2+\varepsilon}(\Omega) $ ($\varepsilon > 0$ and integer $r\geq 1$).
Secondly, for the solution $u$ of \eqref{Eqn:Varu}, there exists a positive constant $\alpha_0$ such that
\begin{equation}\label{Eqn:Coraz}
\boldsymbol{\xi}^{T} \boldsymbol{a}_{\boldsymbol{z}}(u) \boldsymbol{\xi} \geq \alpha_{0}\mid\boldsymbol{\xi}\mid^{2}, \quad \forall \boldsymbol{\xi} \in \mathbb{R}^{2}, \ \boldsymbol{x} \in \bar{\Omega}.
\end{equation}
Finally, $\mathcal{L}^{\prime}(u): {H}_{0}^{1}(\Omega) \rightarrow {H}^{-1}(\Omega)$ is an isomorphism. 
These assumptions guarantee that $u$ is an isolated solution of \eqref{Eqn:Varu}.


Let $\mathcal{T}_h$ be a conforming quasi-uniform triangulation on $\Omega$, where the mesh size $h$ denotes the maximum of the circumscribed circle diameters of element $K \in \mathcal{T}_h$.
By this, any element $K\in \mathcal{T}_h$ is contained in (contains) a circle of radius $\hat{C}_1 h$ (respectively, $\hat{C}_2 h$), where the constant $\hat{C}_1$ and $\hat{C}_2$ do not depend on mesh size $h$, and there is no hanging node on $\mathcal{T}_h$.

The finite element space $V_h$ on $\mathcal{T}_h$ is defined as
\begin{equation*}
V_h = \{ v_h \in H_0^1(\Omega) : v_h\mid_K \in \mathcal{P}_r(K),\ \forall\  K \in \mathcal{T}_h \},
\end{equation*}
where $\mathcal{P}_r(K)$ is the set of polynomials of degree at most integer $r$ on $K$.

Here is the discrete system of \eqref{Eqn:Varu}: Find $u_{h} \in {V}_{h}$ such that
\begin{equation}\label{Eqn:uh}
 A\left(u_{h}, v_h\right)=0, \quad \forall v_h \in {V}_{h}.
\end{equation}

The following lemma presents the well-posedness of the variational problem \eqref{Eqn:uh} and its priori error estimates, which can be found in Lemma 3.2 and Theorem 3.4 of \cite{XuJC96:1759}, respectively.

\begin{lemma}\label{Lem:uh}
Assume $u $ is the solution of problem \eqref{Eqn:Varu}, then when $h$ is small enough, the discrete variational problem \eqref{Eqn:uh} exists a unique solution $u_h \in V_h$, and the following priori error estimate holds 
\begin{eqnarray}\label{Eqn:u-uh1p} 
\left\|u-u_{h}\right\|_{1, p} \lesssim h^{r}, \quad \text { if }\ u \in {W}^{r+1, p}(\Omega),  2 \leq p \leq \infty. 
\end{eqnarray}
\end{lemma}

\section{Iterative two-grid algorithms}\label{Sec:algorithm}

In this section, we present an iterative two-grid algorithm for the variational problems \eqref{Eqn:Varu}. 

Let $\mathcal {T}_h$ and $\mathcal {T}_H$ be two quasi-uniform, conforming and nested mesh in $\Omega$.
Furthermore the mesh size $h$ of $\mathcal{T}_h$ and $H$ of $\mathcal{T}_H$ satisfy, for some $0 < \lambda <1$,
$$
H = O(h^\lambda) \quad \text{and} \quad h < H < 1.
$$

For present the iterative two-grid algorithm, 
 we introduce the form $B(w; v, \chi)$ (induced by $\mathcal{L}^\prime$) by ,  for a fixed $w$ and any $v, \chi \in H_0^1(\Omega)$, 
\begin{equation}\label{Eqn:B}
B(w; v, \chi) = (\boldsymbol{a}_y(w)v, \nabla \chi) + (\boldsymbol{a}_{\boldsymbol{z}}(w)\nabla v, \nabla \chi) + (f_y(w)v,  \chi) + (f_{\boldsymbol{z}}(w)\nabla v,  \chi).
\end{equation}

\begin{remark}\label{Rem:Bbilinear}
The form $B(w; \cdot, \cdot)$ is a bilinear form with fixed $w$. 
\end{remark}

To our knowledge, the two-grid algorithms of strongly nonlinear problems are firstly proposed in \cite{XuJC94:79}.
Here one of two-grid algorithms from Algorithm 3.3 of \cite{XuJC94:79} is given. 
\begin{algorithm}
\label{Alg}

\begin{description}
\item[1.] Find $u_H \in {V}_{H}$, such that
$$
A(u_H, v_H) = 0, \qquad \forall v_H \in V_H.
$$
\item[2.] Find $u^h \in {V}_{h}$, such that
$$
B(u_H; u^h, v_h) =  B(u_H; u_H, v_h) - A(u_H, v_h), \qquad \forall v_h \in V_h.  
$$
\end{description}
\end{algorithm}

\begin{remark}
In the Algorithm \ref{Alg}, we first solve a nonlinear problem in a coarse space $V_H$.
However, because $\mathrm{dim}(V_H)$ is relatively small, the calculated amount of solving a nonlinear problem in $V_H$ is not excessive.
As for the second step of Algorithm \ref{Alg}, noticing that $B(u_H; \cdot, \cdot)$ is a bilinear form with given $u_H$, we simply need to solve a linear problem in $V_h$, for which there are numerous concerning fast algorithms.
\end{remark}

In \cite{XuJC94:79},  Xu had showed that the solution $u^h$ of Algorithm \ref{Alg} could be a good approximation with respect to finite element solution $u_h$ at a low cost, namely,
\begin{equation}\label{Eqn:Xu}
\|u_h -u^h\|_1 \lesssim H^2. 
\end{equation}
Using triangle inequality, \eqref{Eqn:u-uh1p} with $r=1$ and \eqref{Eqn:Xu}, we obtain the error estimate of Algorithm \ref{Alg}, 
\begin{equation}\label{Eqn:XuError}
\| u - u^h\|_1 \leq \| u - u_h \|_1 + \| u_h - u^h\|_1 \lesssim h + H^2. 
\end{equation}

Next, putting the Algorithm \ref{Alg} into a successive fashion, we obtain our iterative two-grid algorithm. 
\begin{algorithm}
\label{algorithm}
Let $u_h^0 = u_H$ be the solution of \eqref{Eqn:uh} in $V_H$.
Assume that $u_h^k\in V_h$ has been obtained, then $u_h^{k+1}\in V_h$ can be obtained by the following two steps.

\textbf{ Step 1.} Find $e_H^k\in V_H$ such that, for any $v_H\in V_H$, 
\begin{eqnarray} \label{Eqn:eHk}
A(u_h^k + e_H^k, v_H)=0.
\end{eqnarray}

\textbf{ Step 2.} Find $u^{k+1}_h\in V_h$ such that, for any $v_h \in V_h$, 
\begin{eqnarray}\label{Eqn:uhk+1}
B(u_h^k + e_H^k; u_h^{k+1}, v_h) = B(u_h^k + e_H^k; u_h^k + e_H^k, v_h) - A(u_h^k + e_H^k, v_h).
\end{eqnarray}
\end{algorithm}

\begin{remark}\label{Rem:VS}
Noticing the uniqueness of finite element solution (See Lemma \ref{Lem:uh}), \eqref{Eqn:uh}, $u_h^0 = u_H$ and \eqref{Eqn:eHk} with $k = 0$, we can see that $e_H^0 = 0$, which means $u_h^0 + e_H^0 = u_H$. 
By observing the the Step 2 of Algorithm \ref{algorithm} and the second step of Algorithm \ref{Alg}, the conclusion is that Algorithm \ref{algorithm} is same with Algorithm \ref{Alg} when $k = 0$. 
\end{remark}

In comparison to \cite{XuJC94:79}, our method is still valid for high order conforming finite elements, whereas \cite{XuJC94:79} only considered piecewise linear finite element space.
Here gives the error estimate of our algorithm (See Theorem \ref{Thm:last}), 
\begin{equation}\label{Eqn:Lin}
\| u- u_h^{k}\|_1 \lesssim h^r + H^{r+k}. 
\end{equation}
Specially, if we choose finite element space $V_h$ as piecewise linear finite element space, i.e. $r = 1$, the error estimate \eqref{Eqn:Lin} of Algorithm \ref{algorithm} could be written as
\begin{equation*}\label{Eqn:ZhanError}
\| u- u_h^{k}\|_1 \lesssim h + H^{1+k}. 
\end{equation*}
To achieve the optimal convergence order, the relationship $h = H^2$ should be satisfied in Algorithm \ref{Alg} (See \eqref{Eqn:XuError}).
But in Algorithm \ref{algorithm}, we could expand the distance between the mesh size $H$ and $h$ by increasing the iteration counts $k$.

\section{Convergence analysis}\label{Sec:convergence}

In this section, we provide the corresponding convergence analysis of Algorithm \ref{algorithm}. 
To this end, we need to introduce some preliminaries based on form $B(w;v, \chi)$ at first. 

\subsection{Some preliminaries based on form $B(w; v, \chi)$}

In this subsection, we present some properties of form $B(w; v, \chi)$ and introduce two discrete Green function.

Firstly, with fixed $w $, by Remark \ref{Rem:Assaf} and Cauchy-Schwarz inequality, it's easy to obtain that the form $B(w; \cdot, \cdot)$ is continuous, i.e., 
\begin{equation}\label{Eqn:Bcontinuous}
B(w; v, \chi)  \lesssim \|v\|_1 \|\chi \|_1, \quad  \forall \  v, \chi \in H_0^1(\Omega). 
\end{equation}

Secondly, we present a lemma which provides the Babu\v{s}ka-Brezzi(BB) conditions of  form $B(\cdot; \cdot, \cdot)$ in $V_h$.
And this lemma can be proved using similar arguments in Lemma 2.2 of \cite{XuJC96:1759}. 
\begin{lemma}\label{Lem:infsupBu}
Assume $u $ is the solution of problem \eqref{Eqn:Varu}, then when $h$ is small enough, we have, for any $w_h \in V_h$, 
\begin{equation}\label{Eqn:infsupBu}
\left\|w_{h}\right\|_{1} \lesssim \sup_{v_h \in V_h} \dfrac{B\left(u;w_h, v_h\right)}{\|v_{h}\|_1}
\quad \text{and} \quad
\left\|w_{h}\right\|_{1} \lesssim \sup_{v_h \in V_h} \dfrac{B\left(u;v_h, w_h\right)}{\|v_{h}\|_1}. 
\end{equation}
\end{lemma}

\begin{proof}
For the solution $u$ of \eqref{Eqn:Varu}, a projection operator $P_h: H_0^1(\Omega) \rightarrow V_h $  is defined by
\begin{equation}\label{Eqn:Ph}
(\boldsymbol{a}_{\boldsymbol{z}}(u) \nabla P_h v, \nabla \chi_h) = (\boldsymbol{a}_{\boldsymbol{z}}(u)  \nabla v, \nabla \chi_h ), \quad \forall v\in H_0^1(\Omega), \chi_h \in V_h.  
\end{equation}
By \eqref{Eqn:Coraz}, we can know that the projection operator $P_h$ is well-defined. 
Taking $ v = v_h \in V_h \subset H_0^1(\Omega)$ and $\chi_h = P_h v_h - v_h$, and using \eqref{Eqn:Coraz}, we could prove that the projection operator $P_h$ is identity operator for space $V_h$. 
Substituting $\chi_h = P_h v$ into \eqref{Eqn:Ph}, and using \eqref{Eqn:Coraz}, Poincar\'{e} inequality, Remark \ref{Rem:Assaf} and Cauchy--Schwarz inequality, it holds that
\begin{eqnarray}\label{Eqn:Phlesssim}
& \| P_hv\|_1 \lesssim \|v\|_1, \quad \forall v \in H_0^1(\Omega) . 
\end{eqnarray}
By using \eqref{Eqn:Coraz}, duality argument and \eqref{Eqn:Phlesssim}, we can obtain  
(See Theorem 3.2.5 in \cite{Ciarlet02Book})
\begin{eqnarray}\label{Eqn:I-Ph}
&\| v-P_hv\|_0 \lesssim h \|v\|_1, \quad \forall v \in H_0^1(\Omega). 
\end{eqnarray}

For any $w_h \in V_h$, $v \in H_0^1(\Omega)$, by \eqref{Eqn:B}, Green formula, \eqref{Eqn:Ph}, Remark \ref{Rem:Assaf}, Cauchy-Schwarz inequality and \eqref{Eqn:I-Ph}, we have
\begin{eqnarray} \nonumber
B(u; w_h, v - P_h v) 
&=& (\boldsymbol{a}_y(u) w_h, \nabla (v - P_h v)) + (\boldsymbol{a}_{\boldsymbol{z}}(u)\nabla w_h, \nabla (v - P_h v))
\\ \nonumber
&& + (f_y(u) w_h,  v - P_h v) + (f_{\boldsymbol{z}}(u)\nabla w_h,  v - P_h v)
%
\\ \nonumber
&=& ( (\nabla \cdot  \boldsymbol{a}_y(u)) w_h,  v - P_h v) + (   \boldsymbol{a}_y(u) \cdot \nabla w_h,  v - P_h v)
\\ \nonumber
&& + (f_y(u) w_h,  v - P_h v) + (f_{\boldsymbol{z}}(u)\nabla w_h,  v - P_h v)
%
\\ \nonumber
&\lesssim& \|w_h\|_1 \|v-P_h v\|_0
%
\\ \label{Eqn:Buwhv-Phv}
&\lesssim& h \|w_h\|_1 \|v\|_1. 
\end{eqnarray}
Noticing that  $\mathcal{L}^{\prime}(u): {H}_{0}^{1}(\Omega) \rightarrow {H}^{-1}(\Omega)$ is an isomorphism, using \eqref{Eqn:Buwhv-Phv} and \eqref{Eqn:Phlesssim}, we obtain that
\begin{eqnarray*}
\|w_h\|_1 &\lesssim& \sup_{v \in H_0^1(\Omega)} \dfrac{B(u; w_h, v)}{\| v\|_1} 
%
\\
&\lesssim&  \sup_{v \in H_0^1(\Omega)} \dfrac{B(u; w_h, v- P_h v)}{\| v\|_1} + \sup_{v \in H_0^1{\Omega}} \dfrac{B(u; w_h,  P_h v)}{\| v\|_1}
\\
&\lesssim& h\|w_h\|_1 + \sup_{v \in H_0^1(\Omega)} \dfrac{B(u; w_h,  P_h v)}{\| P_h v\|_1}.
\end{eqnarray*}
Taking $h$ sufficiently small in the above inequality with projection operator $P_h$ being identity operator for $V_h$, we could obtain the first estimate of \eqref{Eqn:infsupBu}. The proof of the second estimate of \eqref{Eqn:infsupBu} is similar.
\end{proof}

Next, we provide another BB condition of the form $B(\cdot; \cdot, \cdot)$.

\begin{lemma}\label{Lem:infsupBuhkeHk}
Assume $u$ is the solution of \eqref{Eqn:Varu} and $\Psi $ satisfying $\|  u - \Psi \|_{1, \infty} \lesssim H$, then when $H$ is small enough, for any $w_h \in V_h$, it holds that
\begin{equation}\label{Eqn:infsupBuhkeHk}
\left\|w_{h}\right\|_{1} \lesssim \sup_{v_h \in V_h} \dfrac{B\left(\Psi;w_h, v_h\right)}{\|v_{h}\|_1}
\quad \text{and} \quad
\left\|w_{h}\right\|_{1} \lesssim \sup_{v_h \in V_h} \dfrac{B\left(\Psi;v_h, w_h\right)}{\|v_{h}\|_1}. 
\end{equation}
\end{lemma}

\begin{proof}
Using the definition \eqref{Eqn:B} of form $B$, Taylor expansion
%
$
h(y, z) = h(y_0, z_0) + \partial_y h (   \tilde{\theta}_1, \tilde{\theta}_2 )(y-y_0) +   \partial_z h (\tilde{\theta}_1, \tilde{\theta}_2 )(z-z_0),
$
where $\tilde{\theta}_1$ is between $y$ and $y_0$ and $\tilde{\theta}_2$ is between $z$ and $z_0$, 
Remark \ref{Rem:Assaf}, 
and H\"{o}lder inequality, we obtain
\begin{eqnarray}\nonumber
\lefteqn{B(u; w_h, v_h) - B(\Psi; w_h, v_h)}
\\ \nonumber
&=& ((\boldsymbol{a}_y(u) - \boldsymbol{a}_y(\Psi)) w_h, \nabla v_h) + ((\boldsymbol{a}_{\boldsymbol{z}}(u) - \boldsymbol{a}_{\boldsymbol{z}}(\Psi)) \nabla w_h, \nabla v_h) 
\\ \nonumber
&&  + ((f_y(u) - f_y(\Psi)) w_h,  v_h) + ((f_{\boldsymbol{z}}(u) - f_{\boldsymbol{z}}(\Psi))\nabla w_h,  v_h)
%
\\ \nonumber
&=& (\boldsymbol{a}_{yy}(\theta_1) (u-\Psi) w_h, \nabla v_h) + (\boldsymbol{a}_{y \boldsymbol{z}}(\theta_1) \nabla(u-\Psi) w_h, \nabla v_h)
\\ \nonumber
&& + (\boldsymbol{a}_{\boldsymbol{z}y}(\theta_2) (u-\Psi) \nabla w_h, \nabla v_h)  + (\nabla(u-\Psi)^T\boldsymbol{a}_{\boldsymbol{zz}}(\theta_2)  \nabla w_h, \nabla v_h)
\\ \nonumber
&& +(f_{yy}(\theta_3) (u-\Psi) w_h,  v_h) + (f_{y\boldsymbol{z}}(\theta_3) \cdot \nabla(u-\Psi) w_h,  v_h)
\\ \nonumber
&& +(f_{\boldsymbol{z}y}(\theta_4) \cdot  \nabla w_h (u-\Psi),  v_h) + (\nabla(u-\Psi)^T f_{\boldsymbol{zz}}(\theta_4)   \nabla w_h,  v_h)
%
\\ \label{Eqn:Buwhvh}
&\lesssim &   \| u-\Psi\|_{1, \infty} \|w_h\|_1 \|v_h\|_1,
\end{eqnarray}
where $\theta_i\ (i = 1, 2, 3, 4)$ are between $u$ and $\Psi$. 

By Lemma \ref{Lem:infsupBu}, \eqref{Eqn:Buwhvh} and $\| u - \Psi\|_{1, \infty} \lesssim H$, it is obtained that
\begin{eqnarray}\nonumber
\|w_h\|_1 &\lesssim&   \sup_{v_h \in V_h} \dfrac{B(u; w_h, v_h) - B\left(\Psi; w_h, v_h\right)}{\|v_{h}\|_1}
+   \sup_{v_h \in V_h} \dfrac{B\left(\Psi;w_h, v_h\right)}{\|v_{h}\|_1}
%
\\  
\nonumber
&\lesssim&   \| u- \Psi\|_{1, \infty} \|w_h\|_1  +   \sup_{v_h \in V_h} \dfrac{B\left(\Psi;w_h, v_h\right)}{\|v_{h}\|_1}
%
\\  \nonumber
&\lesssim & H \|w_h\|_1  +  \sup_{v_h \in V_h} \dfrac{B\left(\Psi;w_h, v_h\right)}{\|v_{h}\|_1}. 
\end{eqnarray}
Taking $H$ sufficiently small into the above inequality, we can derive the first estimate of \eqref{Eqn:infsupBuhkeHk}. The proof of the second estimate of \eqref{Eqn:infsupBuhkeHk} is similar. 
\end{proof}

\begin{remark}\label{Rem:infsupBuh}
According to \eqref{Eqn:u-uh1p}, Lemma \ref{Lem:infsupBuhkeHk} still holds with replacing $\Psi$ by the finite element solution $u_h$ of \eqref{Eqn:uh}. 
\end{remark}

For more concise notations and the subsequent analysis, we denote
\begin{eqnarray}\label{Eqn:Ek}
&E_k = u_h - u_h^k, 
\\ \label{Eqn:uhk1}
& u_h^{k, 1} = u_h^k + e_H^k, 
\end{eqnarray}
where $u_h$ is the solution of problem \eqref{Eqn:uh} and, $u_h^k$ and $e_H^k$ are given by Algorithm \ref{algorithm}.
It's noted that these notation will be used frequently in the rest of this paper.

\begin{remark}\label{Rem:BBuhk+eHk}
For $k \geq 0$, assume that $E_k$, $u_h^{k, 1}$ and $e_H^k$ are given by \eqref{Eqn:Ek}, \eqref{Eqn:uhk1} and Algorithm \ref{algorithm}, respectively. 
If both $\| E_k\|_{1, \infty} \lesssim H$ and $\| e_H^k\|_{1, \infty} \lesssim H$ are provided, the Lemma \ref{Lem:infsupBuhkeHk} still holds  with replacing $\Psi$ by $u_h^{k, 1}$. 
In fact, using \eqref{Eqn:uhk1}, \eqref{Eqn:Ek}, triangle inequality, \eqref{Eqn:u-uh1p} with $r \geq 1$ and $h < H$, $\| E_k\|_{1, \infty} \lesssim H$ and $\| e_H^k\|_{1, \infty} \lesssim H$, we derive that 
$$
\| u - u_h^{k, 1}\|_{1, \infty}
\leq \|u - u_h \|_{1, \infty} +   \| E_k \|_{1, \infty} +  \| e_H^k \|_{1, \infty}
\lesssim H. 
$$
Therefore, the Lemma \ref{Lem:infsupBuhkeHk} still holds  with  $\Psi = u_h^{k, 1}$. 
\end{remark}

And then, for the finite element solution $u_h$ of \eqref{Eqn:uh} and any fixed $\boldsymbol{x} \in \Omega$, we introduce the Green functions $g_H^{\boldsymbol{x}} \in V_H$, which be defined by
\begin{eqnarray} \label{Eqn:gHz}
&B(u_h; v_H, g_H^{\boldsymbol{x}}) = \partial v_H(\boldsymbol{x}), \quad \forall  v_H \in V_H, 
\end{eqnarray}
where $\partial$ denotes either $\frac{\partial}{\partial x_{1}}$ or $\frac{\partial}{\partial x_{2}}$.
It's easy to see that the Green function $g_H^{\boldsymbol{x}}$ is well-defined by Remark \ref{Rem:infsupBuh}.

Assume $u_h^{k, 1}$ is given by \eqref{Eqn:uhk1}, similarly, for any fixed $\boldsymbol{x} \in \Omega$, we introduce the Green functions $g_h^{k, \boldsymbol{x}} \in V_h$ by
\begin{eqnarray}\label{Eqn:ghz}
&B(u_h^{k, 1}; v_h, g_h^{k, \boldsymbol{x}}) = \partial v_h(\boldsymbol{x}), \quad \forall  v_h \in V_h. 
\end{eqnarray}
By Remark \ref{Rem:BBuhk+eHk}, we also can see that Green function $g_h^{k, \boldsymbol{x}}$ is well-defined.

Here gives some estimates of the above two Green functions $g_H^{\boldsymbol{x}}$ and $g_h^{k, \boldsymbol{x}} $
(See Lemma 3.3 of \cite{ThomeeXu89:553}, or (2.10) and (2.11) of \cite{XuJC96:1759}) 
\begin{equation}\label{Eqn:ghz11}
\|g_H^{\boldsymbol{x}}\|_{1, 1} \lesssim \mid\log H\mid
\quad \text{and} \quad  
\|g_h^{k, \boldsymbol{x}}\|_{1, 1} \lesssim \mid\log h\mid . 
\end{equation}

At last, for any $v \in  H_0^1(\Omega) \cap W^{1, \infty}(\Omega)$, using (3.1.11) of \cite{Ciarlet02Book}, it could be obtained that
\begin{equation}\label{Eqn:1infty}
\| v \|_{1, \infty} \lesssim |   v|_{1, \infty}. 
\end{equation}

\subsection{Error estimate}

In this subsection, we present the convergence analysis of Algorithm \ref{algorithm} by a series of lemmas.

\begin{lemma}\label{Lem:Ek+11}
Assume $u_h^{k, 1}$, $E_{k}$ and $e_H^k$ are given by \eqref{Eqn:uhk1}, \eqref{Eqn:Ek} and Algorithm \ref{algorithm}, respectively, then we have, for any $v_h \in V_h$, 
\begin{eqnarray}\label{Eqn:Buhk111}
&B(u_h^{k, 1}; E_{k+1}, v_h) \lesssim (\| E_k\|_{1, \infty} + \| e_H^k\|_{1, \infty} )( \| E_k \|_{1} +  \|e_H^k\|_{1} ) \|v_h\|_1,
\\ \label{Eqn:Buhk1infty}
&B(u_h^{k, 1}; E_{k+1}, v_h) \lesssim (\| E_{k}\|_{1, \infty}^2 + \|e_H^k \|_{1, \infty}^2) \| v_h\|_{1, 1}. 
\end{eqnarray}
\end{lemma}

\begin{proof}
Using \eqref{Eqn:Ek}, Remark \ref{Rem:Bbilinear}, \eqref{Eqn:uhk+1}, \eqref{Eqn:uh} and \eqref{Eqn:A}, it is obtained that
\begin{eqnarray}\nonumber
B(u_h^{k, 1}; E_{k+1}, v_h) &=& B(u_h^{k, 1}; u_h, v_h) - B(u_h^{k, 1}; u_h^{k+1}, v_h)
%
\\ \nonumber
&=& B(u_h^{k, 1}; u_h, v_h) - B(u_h^{k, 1}; u_h^k + e_H^k, v_h)
+ A(u_h^{k, 1}, v_h) - A(u_h, v_h)
%
\\  \nonumber
&= &    B(u_h^{k, 1}; E_k - e_H^k, v_h)
+  (\boldsymbol{a} (u_h^{k, 1}, \nabla u_h^{k, 1}), v_h)
+ (f(u_h^{k, 1}, \nabla u_h^{k, 1}), v_h) 
\\  \nonumber
&&   - (\boldsymbol{a} (u_h, \nabla u_h), v_h) - (f(u_h, \nabla u_h), v_h)
%
\\ \label{Eqn:BEk+1}
&:=& A_1 - A_2 - A_3,
\end{eqnarray}
where
\begin{eqnarray*}
& A_1 = B(u_h^{k, 1}; E_k - e_H^k, v_h),
\\
& A_2 = (\boldsymbol{a} (u_h, \nabla u_h), v_h) -  (\boldsymbol{a} (u_h^{k, 1}, \nabla u_h^{k, 1}), v_h),
\\  
& A_3 = (f(u_h, \nabla u_h), v_h) - (f(u_h^{k, 1}, \nabla u_h^{k, 1}), v_h).
\end{eqnarray*}

For $A_1$, using the definition \eqref{Eqn:B} of $B$, we have
\begin{eqnarray}\nonumber
A_1 
&=&  
( \boldsymbol{a}_y(u_h^{k, 1})  (E_k - e_H^k), \nabla v_h) 
+ ( \boldsymbol{a}_{\boldsymbol{z}}(u_h^{k, 1}) \nabla (E_k - e_H^k), \nabla v_h)
\\ \label{Eqn:A1}
&& + ( f_y(u_h^{k, 1}) (E_k - e_H^k) ,  v_h) 
+ ( f_{\boldsymbol{z}}(u_h^{k, 1}) \nabla (E_k - e_H^k),  v_h). 
\end{eqnarray}

For $A_2$, using second order Taylor expansion, \eqref{Eqn:uhk1} and \eqref{Eqn:Ek}, we obtain
\begin{eqnarray}\nonumber
A_2 &=&( \boldsymbol{a}_y(u_h^{k, 1})  (E_k - e_H^k), \nabla v_h) 
+ ( \boldsymbol{a}_{\boldsymbol{z}}(u_h^{k, 1}) \nabla (E_k - e_H^k), \nabla v_h) 
\\ \nonumber
&&   + (\boldsymbol{a}_{yy}(\theta_5) (E_k - e_H^k)^2 , \nabla v_h) +  2 (\boldsymbol{a}_{y \boldsymbol{z}}(\theta_5) \nabla (E_k - e_H^k)  (E_k - e_H^k), \nabla v_h) 
\\ \label{Eqn:A2}
&&   + (\nabla(E_k - e_H^k)^T\boldsymbol{a}_{\boldsymbol{zz}}(\theta_5)  \nabla (E_k - e_H^k), \nabla v_h) , 
\end{eqnarray}
where $\theta_5$ is between $u_h$ and $u_h^{k, 1}$.

Similarly for $A_3$, using second order Taylor expansion, \eqref{Eqn:uhk1} and \eqref{Eqn:Ek}, it is obtained that 
\begin{eqnarray}\nonumber
A_3
&=&   ( f_y(u_h^{k, 1}) (E_k - e_H^k) ,  v_h) + ( f_{\boldsymbol{z}}(u_h^{k, 1}) \nabla (E_k - e_H^k),  v_h) 
\\ \nonumber
&&  +(f_{yy}(\theta_6) (E_k - e_H^k)^2,  v_h) + 2 (f_{y\boldsymbol{z}}(\theta_6) \cdot \nabla(E_k - e_H^k) (E_k - e_H^k),  v_h) 
\\ \label{Eqn:A3}
&&   + (\nabla(E_k - e_H^k)^T f_{\boldsymbol{zz}}(\theta_6)   \nabla (E_k - e_H^k),  v_h), 
\end{eqnarray}
where $\theta_6$ is between $u_h$ and $u_h^{k, 1}$.

Noticing the sum of the first order derivative items about $\boldsymbol{a}(\cdot, \cdot, \cdot)$ and $f(\cdot, \cdot, \cdot)$ in \eqref{Eqn:A2} and \eqref{Eqn:A3} exactly equal $A_1$. 
Substituting \eqref{Eqn:A1}, \eqref{Eqn:A2} and \eqref{Eqn:A3} into \eqref{Eqn:BEk+1}, it's obtained that
\begin{eqnarray} \nonumber
B(u_h^{k, 1}; E_{k+1}, v_h)
&=  &  - (\boldsymbol{a}_{yy}(\theta_5) (E_k - e_H^k)^2 , \nabla v_h) 
-  2 (\boldsymbol{a}_{y \boldsymbol{z}}(\theta_5) \nabla (E_k - e_H^k)  (E_k - e_H^k), \nabla v_h)
\\ \nonumber
&& 
   - (\nabla(E_k - e_H^k)^T\boldsymbol{a}_{\boldsymbol{zz}}(\theta_5)  \nabla (E_k - e_H^k), \nabla v_h)
 - (f_{yy}(\theta_6) (E_k - e_H^k)^2,  v_h) 
\\ \nonumber
&& 
-  2 (f_{y\boldsymbol{z}}(\theta_6) \cdot \nabla(E_k - e_H^k) (E_k - e_H^k),  v_h)
\\  \label{Eqn:Buhk1Ek+1lesssim}
&& - (\nabla(E_k - e_H^k)^T f_{\boldsymbol{zz}}(\theta_6)   \nabla (E_k - e_H^k),  v_h). 
\end{eqnarray}
Applying Remark \ref{Rem:Assaf}, H\"{o}lder inequality and triangle inequality into \eqref{Eqn:Buhk1Ek+1lesssim}, we could obtain
\begin{eqnarray*}
B(u_h^{k, 1}; E_{k+1}, v_h) &\lesssim& \| E_k - e_H^k\|_{1, \infty} \| E_k - e_H^k\|_{1} \|v_h\|_1
\\
&\leq& (\| E_k\|_{1, \infty} + \| e_H^k\|_{1, \infty} )( \| E_k \|_{1} +  \|e_H^k\|_{1} ) \|v_h\|_1,
\end{eqnarray*}
which completes the proof of \eqref{Eqn:Buhk111}. 
Similarly,  we could obtain \eqref{Eqn:Buhk1infty} by \eqref{Eqn:Buhk1Ek+1lesssim}.
\end{proof}

\begin{lemma}\label{Lem:Buhk1eHkvH}
Assume that $u_h^{k, 1}$, $e_H^k$ and $E_k$ are defined by \eqref{Eqn:uhk1}, Algorithm \ref{algorithm} and \eqref{Eqn:Ek}, respectively, then we have
\begin{equation}\label{Eqn:Buhk1eHkvH}
B(u_h^{k, 1}; e_H^k, v_H) \lesssim (\|E_k\|_1 + \|E_{k+1}\|_1 ) \|v_H\|_1, \quad \forall v_H \in V_H. 
\end{equation}
\end{lemma}

\begin{proof}
Taking $v_h = v_H$ into \eqref{Eqn:uhk+1} and using \eqref{Eqn:eHk}, we obtain
$$
B(u_h^{k, 1}; u_h^{k+1}, v_H) = B(u_h^{k, 1}; u_h^{k} + e_H^{k}, v_H). 
$$
Rewriting the the above equation with Remark \ref{Rem:Bbilinear}, and then using \eqref{Eqn:Ek}, \eqref{Eqn:Bcontinuous} and triangle inequality, we have
\begin{eqnarray} \nonumber
B(u_h^{k, 1};  e_H^{k}, v_H) &=& B(u_h^{k, 1}; u_h^{k+1} - u_h^{k}, v_H)
%
\\ \nonumber
&=& B(u_h^{k, 1}; u_h^{k+1} - u_h + u_h - u_h^{k}, v_H)
\\ \nonumber
&=& B(u_h^{k, 1}; E_k  - E_{k+1}, v_H)
%
\\  \nonumber
&\lesssim & \| E_{k} - E_{k+1}\|_1 \|v_H\|_1
%
\\ \nonumber
&\leq & \left( \| E_{k}\|_1  + \|E_{k+1}\|_1 \right) \|v_H\|_1, 
\end{eqnarray}
which completes the proof. 
\end{proof}

\begin{lemma}\label{Lem:last}
Assume that $E_k$ and $e_H^k$ are given by \eqref{Eqn:Ek} and  Algorithm \ref{algorithm}, respectively, and $r \geq 1$, when $h$ is small enough, then for any integer $k\geq 1$, 
\begin{eqnarray}\label{Eqn:last}
\|E_{k}\|_1 \lesssim H^{r+k}, 
\quad \|E_{k}\|_{1, \infty} \lesssim \mid\log h\mid H^2, 
\quad \|e_H^{k}\|_{1, \infty} \lesssim  H, 
\quad \|e_H^k\|_1 \lesssim H^{r+k}. 
\end{eqnarray}
\end{lemma}

\begin{proof}
Here we use mathematical induction to prove that \eqref{Eqn:last} is true.

By \eqref{Eqn:eHk}, $u_h^0 = u_H$, \eqref{Eqn:uh} and the uniqueness of finite element solution (See Lemma \ref{Lem:uh}), it's could be seen that $e_H^0 = 0$. 

Making use of triangle inequality, \eqref{Eqn:u-uh1p} and $h\leq H$, we have
\begin{eqnarray} \label{Eqn:E01}
&\|E_0\|_1 \leq \|u-u_h\|_1 + \| u - u_H\|_1 \lesssim h^r + H^ r \leq H^r, 
\\ \label{Eqn:E01infty}
&\|E_0\|_{1, \infty} \leq \|u-u_h\|_{1, \infty} + \| u - u_H\|_{1, \infty} 
\lesssim  h^r +  H^r
\leq H^r.
\end{eqnarray}

Next, we will prove \eqref{Eqn:last} is true when $k = 1$. 

(i) For $\|E_{1}\|_1 \lesssim H^{r+1}$. 
Noticing that $r\geq 1$ and $e_H^0 = 0$,
and using \eqref{Eqn:E01infty}, we have $\|E_0\|_{1, \infty} \lesssim H$ and $\|e_H^0\|_{1, \infty} \lesssim H$ , which could derive the BB condition of form $B(u_h^{0, 1}; \cdot, \cdot)$ (See Remark \ref{Rem:BBuhk+eHk}). 
Using the BB condition of form $B(u_h^{0, 1}; \cdot, \cdot)$, \eqref{Eqn:Buhk111}, \eqref{Eqn:E01infty}, $e_H^k = 0$,
\eqref{Eqn:E01}, $r \geq 1$ and $H < 1$, it's obtained that
\begin{eqnarray} \nonumber
\|E_1\|_1 
&\lesssim&  \sup\limits_{v_h \in V_h} \dfrac{B(u_h^{0, 1}; E_1, v_h)}{\|v_h\|_1}
%
\\ \nonumber
&\lesssim&  
 (\|E_0\|_{1, \infty} + \|e_H^0\|_{1, \infty})(\|E_0\|_1 + \|e_H^0\|_1) 
 %
\\   \nonumber
&\lesssim& (H^{r} + 0)( H^{r} + 0)
%
\\ \label{Eqn:E1}
& \lesssim& H^{r+1}.
\end{eqnarray}

(ii)  For $\|E_{1}\|_{1, \infty} \lesssim \mid\log h\mid H^2$.
For $k=0$ and any fixed $\boldsymbol{x} \in \Omega$, taking $v_h = E_1$ into \eqref{Eqn:ghz}, using \eqref{Eqn:Buhk1infty}, \eqref{Eqn:E01infty}, $e_H^0 = 0$, \eqref{Eqn:ghz11}, $r \geq 1$ and $H < 1$, we obtain 
\begin{eqnarray*}
\partial E_{1}(\boldsymbol{x}) & = & B(u_h^{0, 1}; E_{1}, g_h^{0, \boldsymbol{x}})
%
\\
&\lesssim& (\| E_0 \|_{1, \infty}^2  + \|e_H^0\|_{1, \infty}^2) \|g_h^{0, \boldsymbol{x}}\|_{1, 1}
%
\\
&\lesssim& (H^{2r} +0) \mid\log h\mid
%
\\
&\lesssim& \mid\log h\mid H^2.  
\end{eqnarray*}
Further using the arbitrariness of $\boldsymbol{x}$ and \eqref{Eqn:1infty}, we derive that
$$
\|E_{1}\|_{1, \infty} \lesssim \mid\log h\mid H^2. 
$$

(iii) For $\|e_H^1\|_{1, \infty} \lesssim H$. 
%
Using $\| E_1\|_{1, \infty} \lesssim \mid\log h\mid H^2$ and Lemma \ref{Lem:eHk1infty} (The specific content of lemma and proof are referred to Appendix), we obtain
$$
\|e_H^1\|_{1, \infty} \lesssim H. 
$$

(iv) For $\|e_H^1\|_1  \lesssim H^{r+1} $.
Noticing that $\|E_1\|_{1, \infty} \lesssim \mid\log h\mid H^2$ and $\|e_H^1\|_{1, \infty} \lesssim H$ are satisfied, therefore the BB condition of form $B(u_h^{1, 1}; \cdot, \cdot)$ holds (See Remark \ref{Rem:BBuhk+eHk}).
Using the BB condition of $B(u_h^{1, 1}; \cdot, \cdot)$ and \eqref{Eqn:Buhk111}, it's obtained that
\begin{eqnarray} \nonumber
\|E_2\|_1 
&\lesssim& \sup\limits_{v_h} \dfrac{B(u_h^{1, 1}; E_2, v_h)}{\|v_h\|_1}
%
\\ \label{Eqn:E21}
&\lesssim& (\| E_1\|_{1, \infty} + \| e_H^1\|_{1, \infty} )( \| E_1 \|_{1} +  \|e_H^1\|_{1} ). 
\end{eqnarray}

Using the BB condition of form $B(u_h^{1, 1}; \cdot, \cdot)$, \eqref{Eqn:Buhk1eHkvH} with $k = 1$, \eqref{Eqn:E21}, $\|E_1\|_{1, \infty} \lesssim \mid\log h\mid H^2$ and $\|e_H^1\|_{1, \infty} \lesssim H$, we have
\begin{eqnarray*}
\|e_H^1\|_1 
&\lesssim& \sup\limits_{v_H \in V_H} \dfrac{B(u_h^{1, 1}; e_H^1, v_H)}{\|v_H\|_1}
%
\\
&\lesssim& \| E_1\|_1  + \|E_2\|_1
%
\\
&\lesssim& \| E_1\|_1  + (\| E_1\|_{1, \infty} + \| e_H^1\|_{1, \infty} ) ( \| E_1 \|_{1} +  \|e_H^1\|_{1} )
%
\\
&\lesssim& \| E_1\|_1  + (\mid\log h\mid H^2 + H) ( \| E_1 \|_{1} +  \|e_H^1\|_{1} ). 
\end{eqnarray*}
Taking $H$ be small enough in the above inequality, and using \eqref{Eqn:E1}, it's obtained that
$$
\|e_H^1\|_1  \lesssim \| E_1\|_1 \lesssim H^{r+1}.
$$

We assume \eqref{Eqn:last} is true when $k=l$, i.e., 
\begin{eqnarray}\label{Eqn:lastk=l}
\|E_l\|_1 \lesssim  H^{r+l}, 
\quad \|E_l\|_{1, \infty} \lesssim \mid\log h\mid H^{2}, 
\quad \|e_H^l\|_{1, \infty} \lesssim  H, 
\quad \|e_H^l\|_1  \lesssim H^{r+l}. 
\end{eqnarray}
Next, we will prove \eqref{Eqn:last} also holding when $k=l+1$.

(i) For $\|E_{l+1}\|_1 \lesssim H^{r+l+1}$. 
Noticing that  $\|E_l\|_{1, \infty} \lesssim \mid\log h\mid H^2 $ and $\|e_H^l\|_{1, \infty} \lesssim H$ are satisfied, therefore the BB condition of form $B(u_h^{l, 1}; \cdot, \cdot)$ holds (See Remark \ref{Rem:BBuhk+eHk}). 
Using the BB condition of form $B(u_h^{l, 1}; \cdot, \cdot)$, \eqref{Eqn:Buhk111}, \eqref{Eqn:lastk=l} and $H < 1$, we obtain
\begin{eqnarray} \nonumber
\|E_{l+1}\|_1 
&\lesssim&  \sup\limits_{v_h \in V_h} \dfrac{B(u_h^{l, 1}; E_{l+1}, v_h)}{\|v_h\|_1}
%
\\ \nonumber
&\lesssim&  
 (\|E_l\|_{1, \infty} + \|e_H^l\|_{1, \infty})(\|E_l\|_1 + \|e_H^l\|_1) 
 %
\\   \nonumber
&\lesssim& (\mid\log h\mid H^2 + H) (H^{r+l} + H^{r+l})
%
\\ \label{Eqn:El+1}
& \lesssim& H^{r+l+1}.
\end{eqnarray}

(ii) For $\|E_{l+1}\|_{1, \infty} \lesssim \mid\log h\mid H^2$.
Taking $v_h = E_{l+1}$ into \eqref{Eqn:ghz} with $k=l$, using \eqref{Eqn:Buhk1infty}, \eqref{Eqn:lastk=l} and \eqref{Eqn:ghz11}, we obtain 
\begin{eqnarray*}
\partial E_{l+1}(\boldsymbol{x}) & = & B(u_h^{l, 1}; E_{l+1}, g_h^{l, \boldsymbol{x}})
%
\\
&\lesssim& (\| E_l \|_{1, \infty}^2 +\| e_H^l\|_{1, \infty}^2 )\|g_h^{l, \boldsymbol{x}}\|_{1, 1}
%
\\
&\lesssim& ( \mid\log h\mid^2 H^{4} +  H^{2}) \mid\log h\mid 
%
\\
&\lesssim& \mid\log h\mid H^2, 
\end{eqnarray*}
which combining the arbitrariness of $\boldsymbol{x}$ and \eqref{Eqn:1infty}, it could be derived that
$$
\|E_{l+1}\|_{1, \infty} \lesssim \mid\log h\mid H^2. 
$$

(iii) For $\|e_H^{l+1}\|_{1, \infty} \lesssim H$. Using $\| E_{l+1}\|_{1, \infty} \lesssim \mid\log h\mid H^2$ and Lemma \ref{Lem:eHk1infty}, we obtain 
\begin{eqnarray}\label{Eqn:eHl1infty}
\|e_H^{l+1}\|_{1, \infty} \lesssim H. 
\end{eqnarray}

(iv) For $\|e_H^{l+1}\|_1 \lesssim H^{r+l+1}$. 
Noticing that $\|E_{l+1}\|_{1, \infty} \lesssim \mid\log h\mid H^2$ and $\|e_H^{l+1}\|_{1, \infty} \lesssim H$ are satisfied, therefore the BB condition of form $B(u_h^{l+1, 1}; \cdot, \cdot)$ holds (See Remark \ref{Rem:BBuhk+eHk}).
Using the BB condition of form $B(u_h^{l+1, 1}; \cdot, \cdot)$, \eqref{Eqn:Buhk111}, it's obtained that
\begin{eqnarray} \nonumber
\|E_{l+2}\|_1 
&\lesssim& \sup\limits_{v_h} \dfrac{B(u_h^{l+1, 1}; E_{l+2}, v_h)}{\|v_h\|_1}
%
\\ \label{Eqn:El+21}
&\lesssim& (\| E_{l+1}\|_{1, \infty} + \| e_H^{l+1}\|_{1, \infty} )( \| E_{l+1} \|_{1} +  \|e_H^{l+1}\|_{1} ). 
\end{eqnarray}

Using the BB condition of form $B(u_h^{l+1, 1}; \cdot, \cdot)$, \eqref{Eqn:Buhk1eHkvH} with $k = l+1$, \eqref{Eqn:El+21}, $\|E_{l+1}\|_{1, \infty} \lesssim \mid\log h\mid H^2$ and $\|e_H^{l+1}\|_{1, \infty} \lesssim H$, we have
\begin{eqnarray*}
\|e_H^{l+1}\|_1 
&\lesssim& \sup\limits_{v_H \in V_H} \dfrac{B(u_h^{l+1, 1}; e_H^{l+1}, v_H)}{\|v_H\|_1}
%
\\
&\lesssim& \| E_{l+1}\|_1  + \|E_{l+2}\|_1
%
\\
&\lesssim& \| E_{l+1}\|_1  + (\| E_{l+1}\|_{1, \infty} + \| e_H^{l+1}\|_{1, \infty} )( \| E_{l+1} \|_{1} +  \|e_H^{l+1}\|_{1} )
%
\\
&\lesssim& \| E_{l+1}\|_1  + (\mid\log h\mid  H^2 + H) ( \| E_{l+1} \|_{1} +  \|e_H^{l+1}\|_{1} ). 
\end{eqnarray*}
Taking $H$ be small enough in the above inequality and using \eqref{Eqn:El+1}, it's obtained that
$$
\|e_H^{l+1}\|_1    \lesssim \| E_{l+1}\|_1   \lesssim H^{r+l+1}.
$$
By mathematical induction, the conclusion is obtained.
\end{proof}

\begin{remark}
Although we just use the estimation $ \|E_{k+1}\|_1 \lesssim H^{r+k+1}$ in our main result (See Theorem \ref{Thm:last}), the availability of $\|E_{k+1}\|_1 \lesssim H^{r+k+1}$ requires the support of 
$\|E_{k}\|_{1, \infty} \lesssim \mid\log h\mid  H^2$, 
$\|e_H^{k}\|_{1, \infty} \lesssim  H$ 
and  $\|e_H^k\|_1 \lesssim H^{r+k}$. 
\end{remark}


Here gives the main result of this paper.

\begin{theorem}\label{Thm:last}
Assume that $u$ is the solution of \eqref{Eqn:Varu} and $u_h^k$ is given by Algorithm \ref{algorithm}, then we have
\begin{equation}\label{Eqn:L}
\| u- u_h^{k}\|_1 \lesssim h^r + H^{r+k}. 
\end{equation}
\end{theorem}


\begin{proof}
Using triangle inequality, \eqref{Eqn:Ek}, \eqref{Eqn:u-uh1p} and Lemma \ref{Lem:last}, we could obtain that
$$
\| u- u_h^{k}\|_1 \leq \| u- u_h\|_1 + \| E_k\|_1 \lesssim  h^r + H^{r+k},
$$
which completes the proof. 
\end{proof}

\section{Numerical experiments}\label{Sec:numerical}

In this section, we present some numerical experiments to show the efficiency of the proposed iterative two-grid algorithm.
We implemented these experiments by the software package FEALPy of programming language Python \cite{WeiHYHuangYQFEALPy}.
Specially in the Step 1 of Algorithm \ref{algorithm}, we solve the nonlinear systems by Newton iteration methods with relative residual $10^{-8}$.

We adopt the following mean curvature flow problem as our model problem:
\begin{eqnarray*}
- \nabla \cdot \left( \dfrac{\nabla u}{(1+ \mid\nabla u\mid^2)^{1/2}} \right)=g \text { in } \Omega, \quad u=0 \text { on } \partial \Omega, 
\end{eqnarray*}
where the computational domain $\Omega = (0, 1) \times (0, 1)$, the exact solution  $u= x(1-x)^2y(1-y)\mathrm{e}^x$, and $g$  is so chosen according to the exact solution.

Firstly, we choose conforming piecewise linear finite element space as $V_h$, namely choose $r=1$. 
According to Theorem \ref{Thm:last}, we should keep $h^r = H^{r+k}$ hold in order to achieve the optimal convergence order.
Therefore in Table \ref{Tab:5-1}, we present some numerical results in different mesh size with $h = H^2$ for $k=1$. 
In this case, our algorithm is same with Algorithm \ref{Alg} (See Remark \ref{Rem:VS}). 
Furthermore, we could observe that $ \| u - u_h^1 \|_1*\max\{H^{2},h\}^{-1}$ are stable in Table \ref{Tab:5-1}, which agrees with \eqref{Eqn:L} in Theorem \ref{Thm:last}.
\begin{table}[htbp]
	\centering\caption{$k=1, r=1$}
	\label{Tab:5-1}\vskip 0.1cm
	\begin{tabular}{{ccccc}}\hline
		$H$   &$h$    &$\| u - u_h^1 \|_1$  &$ \| u - u_h^1 \|_1*\max\{H^{2},h\}^{-1}$ \\ \hline
1/8  & 1/64  & 3.47E-03    & 0.221929 \\
1/9  & 1/81  & 2.74E-03    & 0.221953 \\
1/10 & 1/100 & 2.22E-03    & 0.221967 \\
1/11 & 1/121 & 1.83E-03    & 0.221977 \\
1/12 & 1/144 & 1.54E-03    & 0.221983
\\ \hline
	\end{tabular}
\end{table}

And then, we increase the iterative counts $k$ to expand the distance between $H$ and $h$, which is shown in Tables \ref{Tab:5-2} and \ref{Tab:5-3}. 
We also observe that $ \| u - u_h^1 \|_1*\max\{H^{1+k},h\}^{-1}$ are stable. 
\begin{table}[htbp]
	\centering\caption{$k=2, r=1$}
	\label{Tab:5-2}\vskip 0.1cm
	\begin{tabular}{{ccccc}}\hline
		$H$   &$h$    &$\| u - u_h^2 \|_1$  &$ \| u - u_h^2 \|_1*\max\{H^{3},h\}^{-1}$ \\ \hline
1/2 & 1/8   & 2.73E-02 & 0.218321 \\
1/3 & 1/27  & 8.20E-03 & 0.221524 \\
1/4 & 1/64  & 3.47E-03 & 0.221784 \\
1/5 & 1/125 & 1.77E-03 & 0.221826 \\
1/6 & 1/216 & 1.03E-03 & 0.221836  \\
\hline
	\end{tabular}
\end{table}
\begin{table}[htbp]
	\centering\caption{$k=3, r=1$}
	\label{Tab:5-3}\vskip 0.1cm
	\begin{tabular}{{ccccc}}\hline
		$H$   &$h$    &$\| u - u_h^3 \|_1$  &$ \| u - u_h^3 \|_1*\max\{H^{4},h\}^{-1}$ \\ \hline
1/2 & 1/16  & 1.38E-02 & 0.220944 \\
1/3 & 1/81  & 2.74E-03 & 0.221805 \\
1/4 & 1/256 & 8.67E-04 & 0.221837 \\
\hline
	\end{tabular}
\end{table}

At last, we implement similar numerical experiments for high order finite element space with $r =2$ and $r =3$ in Tables \ref{Tab:5-4}-\ref{Tab:5-9}. 
By observation, all these numerical experiments are in support of \eqref{Eqn:L} in Theorem \ref{Thm:last}. 
\begin{table}[htbp]
	\centering\caption{$k=1, r=2$}
	\label{Tab:5-4}\vskip 0.1cm
	\begin{tabular}{{ccccc}}\hline
		$H$   &$h$    &$\| u - u_h^1 \|_1$  &$ \| u - u_h^1 \|_1*\max\{H^{3},h^2\}^{-1}$ \\ \hline
1/4  & 1/8   & 2.57E-03 & 0.164578 \\
1/9  & 1/27  & 2.30E-04 & 0.167320                  \\
1/16 & 1/64  & 4.09E-05 & 0.167553                  \\
1/25 & 1/125 & 1.07E-05 & 0.167590                   \\
1/36 & 1/216 & 3.59E-06 & 0.167599  					\\
\hline
	\end{tabular}
\end{table}
\begin{table}[!h]
	\centering\caption{$k=2, r=2$}
	\label{Tab:5-5}\vskip 0.1cm
	\begin{tabular}{{ccccc}}\hline
		$H$   &$h$    &$\| u - u_h^2 \|_1$  &$ \| u - u_h^2 \|_1*\max\{H^{4},h^2\}^{-1}$ \\ \hline
1/8  & 1/64  & 4.09E-05 & 0.167552 \\
1/9  & 1/81  & 2.55E-05 & 0.167571 \\
1/10 & 1/100 & 1.68E-05 & 0.167582  \\
1/11 & 1/121 & 1.14E-05 & 0.167589 \\
1/12 & 1/144 & 8.08E-06 & 0.167593 \\
\hline
	\end{tabular}
\end{table}
\begin{table}[!h]
	\centering\caption{$k=3, r=2$}
	\label{Tab:5-6}\vskip 0.1cm
	\begin{tabular}{{ccccc}}\hline
		$H$   &$h$    &$\| u - u_h^3 \|_1$  &$ \| u - u_h^3 \|_1*\max\{H^{5},h^2\}^{-1}$ \\ \hline
1/5 & 1/55  & 5.54E-05 & 0.167534 \\
1/6 & 1/90  & 2.07E-05 & 0.160874 \\
1/7 & 1/126 & 1.06E-05 & 0.167590\\
1/8 & 1/184 & 4.95E-06 & 0.162211 \\
1/9 & 1/243 & 2.84E-06 & 0.167600  \\
\hline
	\end{tabular}
\end{table}
\begin{table}[!h]
	\centering\caption{$k=1, r=3$}
	\label{Tab:5-7}\vskip 0.1cm
	\begin{tabular}{{ccccc}}\hline
		$H$   &$h$    &$\| u - u_h^1 \|_1$  &$ \| u - u_h^1 \|_1*\max\{H^{4},h^3\}^{-1}$ \\ \hline
1/8  & 1/16  & 1.83E-05 & 0.075054 \\
1/27 & 1/81  & 1.40E-07 & 0.074662 \\
1/64 & 1/256 & 4.44E-09 & 0.074543 \\
\hline
	\end{tabular}
\end{table}
\begin{table}[!h]
	\centering\caption{$k=2, r=3$}
	\label{Tab:5-8}\vskip 0.1cm
	\begin{tabular}{{ccccc}}\hline
		$H$   &$h$    &$\| u - u_h^2 \|_1$  &$ \| u - u_h^2 \|_1*\max\{H^{5},h^3\}^{-1}$ \\ \hline
1/8  & 1/32 & 2.28E-06                     & 0.074870    \\
1/9  & 1/36 & 1.60E-06 					   & 0.074838   \\
1/10 & 1/40 & 1.17E-06                     & 0.074810  \\
1/11 & 1/55 & 4.49E-07                     & 0.072343   \\
1/12 & 1/60 & 3.46E-07                     & 0.074716   \\
\hline
	\end{tabular}
\end{table}
\begin{table}[!h]
	\centering\caption{$k=3, r=3$}
	\label{Tab:5-9}\vskip 0.1cm
	\begin{tabular}{{ccccc}}\hline
		$H$   &$h$    &$\| u - u_h^3 \|_1$  &$ \| u - u_h^3 \|_1*\max\{H^{6},h^3\}^{-1}$ \\ \hline
1/8  & 1/64  & 2.85E-07 & 0.074703 \\
1/9  & 1/81  & 1.40E-07 & 0.074662  \\
1/10 & 1/100 & 7.46E-08 & 0.074630 \\
1/11 & 1/121 & 4.21E-08 & 0.074606 \\
1/12 & 1/144 & 2.50E-08 & 0.074588  \\
\hline
	\end{tabular}
\end{table}

\noindent {\bf Acknowledgments.} The work of the first and second authors were partially funded by the Science and Technology Development Fund, Macau SAR (Nos. 0070/2019/A2, 0031/2022/A1).
The third author was supported by the National Natural Science Foundation of China (Grant No. 11901212). The third and fourth authors are also supported by the National Natural Science Foundation of China (Grant No. 12071160).

\section*{Appendix}\label{secA1}
\setcounter{equation}{0}
\renewcommand\theequation{A.\arabic{equation}}

The purpose of this appendix is to provide the proof of Lemma \ref{Lem:eHk1infty}.

\begin{lemma}\label{Lem:eHk1infty}
Assume $e_H^k$ is given in \eqref{Eqn:eHk} and $\| E_k\|_{1, \infty} \lesssim \mid\log h\mid H^2 $, when $H$ is small enough, it holds that
\begin{equation}\label{Eqn:eHk1infty}
\|e_H^k\|_{1, \infty} \lesssim H.
\end{equation}
\end{lemma}

Before we present the proof of Lemma \ref{Lem:eHk1infty}, we need to introduce some preliminaries and lemmas.

For the finite element solution $u_h$ of \eqref{Eqn:uh}, we introduce a projection operator $\hat{P}_H: H_0^1(\Omega) \rightarrow V_H$, which be defined by, 
\begin{equation}\label{Eqn:hatPH}
B(u_h; \hat{P}_H w, v_H) = B(u_h;  w, v_H), \quad \forall  w \in H_0^1(\Omega), v_H \in V_H.
\end{equation}
It's easy to derive that $\hat{P}_H$ is well-defined by the BB-conditions of form $B(u_h; \cdot, \cdot)$ which could be obtained by Remark \ref{Rem:infsupBuh}. 
Furthermore, the projection operator $\hat{P}_H$ satisfies the following estimate
\begin{equation}\label{Eqn:hatPH1infty}
\|\hat{P}_H w \|_{1, \infty} \lesssim \mid\log H \mid  \|w\|_{1, \infty}, \quad \forall w \in W^{1, \infty}(\Omega).
\end{equation}
In fact, taking $v_H = \hat{P}_H w$ in \eqref{Eqn:gHz}, and using \eqref{Eqn:hatPH}, \eqref{Eqn:B}, 
Remark \ref{Rem:Assaf}, H\"{o}lder inequality and \eqref{Eqn:ghz11}, we obtain
\begin{eqnarray*}
\partial \hat{P}_H w ({\boldsymbol{x}}) &=& B(u_h;  \hat{P}_H w, g_H^{\boldsymbol{x}}) 
%
\\
&=& B(u_h;  w, g_H^{\boldsymbol{x}}) 
%
\\
&=& (\boldsymbol{a}_y(u_h)w, \nabla g_H^{\boldsymbol{x}}) 
+ (\boldsymbol{a}_{\boldsymbol{z}}(u_h)\nabla w, \nabla g_H^{\boldsymbol{x}}) 
+ (f_y(u_h)w,  g_H^{\boldsymbol{x}}) 
+ (f_{\boldsymbol{z}}(u_h)\nabla w,  g_H^{\boldsymbol{x}})
%
\\
&\lesssim& \| w\|_{1, \infty} \|g_H^{\boldsymbol{x}}\|_{1, 1}
%
\\
&\lesssim& \mid\log H \mid  \| w\|_{1, \infty}.  
\end{eqnarray*}
Finally using of the arbitrariness of ${\boldsymbol{x}}$ and \eqref{Eqn:1infty}, we could obtain \eqref{Eqn:hatPH1infty}. 

By Taylor expansion, we have (the detailed proof can be found in Lemma 3.1 of \cite{XuJC96:1759})
\begin{equation}\label{Eqn:A-A}
A(w, \chi) = A(v, \chi) + B(v; w-v, \chi) + R(\eta; v, w, \chi),  \quad \forall w , v , \chi \in H_0^1(\Omega), 
\end{equation}
where the forms $A(\cdot, \cdot)$ and $B(\cdot; \cdot, \cdot)$ are given by \eqref{Eqn:A} and \eqref{Eqn:B}, respectively, $\eta = v+ t(w - v)$ and
\begin{eqnarray*}
R(\eta; v, w, \chi)  &=& \int_0^1 \Big[
(\boldsymbol{a}_{yy} (\eta) (v-w)^2, \nabla \chi)
+ 2 (\boldsymbol{a}_{y\boldsymbol{z}} (\eta) \nabla (v-w) (v-w), \nabla \chi)
\\
&&
+( \nabla (v-w)^T \boldsymbol{a}_{\boldsymbol{zz}} (\eta) \nabla (v-w) , \nabla \chi)
 + (f_{yy} (\eta) (v-w)^2,  \chi)
\\
&&
+ 2 (f_{y\boldsymbol{z}} (\eta) \cdot \nabla (v-w) (v-w),  \chi)
\\
&&
+( \nabla (v-w)^T f_{\boldsymbol{zz}} (\eta) \nabla (v-w) ,  \chi)  \Big] (1-t)\mathrm{d}t. 
\end{eqnarray*}

For the proof of Lemma \ref{Lem:eHk1infty}, we introduce a operator $\Phi$ as follow.
Assume $u_h$ is the solution of \eqref{Eqn:uh}, $E_k$, $R$, $u_h^k$ are given in \eqref{Eqn:Ek}, \eqref{Eqn:A-A} and  Algorithm \ref{algorithm}, respectively, 
we defined operator $\Phi: V_H \rightarrow V_H$ by, for any $w_H \in V_H$, 
%
%
\begin{equation}\label{Eqn:BPhiwH}
B(u_h; \Phi (w_H) , v_H)
=  B(u_h; E_k  , v_H) 
 - R(\eta; u_h, u_h^k + w_H, v_H), \ \forall v_H \in V_H,  
\end{equation}
where $\eta = u_h + t(w_H - E_k)$. 
By the BB-conditions of form $B(u_h; \cdot, \cdot)$ (See Remark \ref{Rem:infsupBuh}), it's easy to prove that operator $\Phi$ is well-defined.

We define a space
\begin{equation}\label{Eqn:QH}
Q_H = \{ v_H \in V_H : \| v_H - \hat{P}_H E_k\|_{1, \infty} \leq H\}, 
\end{equation}
where $\hat{P}_H$ is a projection operator defined by \eqref{Eqn:hatPH}. 
Since $Q_H$ is a finite dimensional space, it's easy to see that $Q_H$ is a non-empty compact convex subset.

Next, we will use Brouwer fixed point theorem to prove that \eqref{Eqn:BPhiwH} has a fixed point $\bar{w}_H$ in $Q_H$.

\begin{lemma}\label{Lem:subset}
Assume $\|E_k\|_{1, \infty} \lesssim \mid\log h\mid H^2$, then when $H$ is small enough, we have $\Phi(Q_H) \subset Q_H$. 
\end{lemma}

\begin{proof}
For any $w_H \in Q_H$, $v_H \in V_H$, 
rewriting \eqref{Eqn:BPhiwH} with \eqref{Eqn:hatPH}, we have
\begin{equation}\label{Eqn:BuhPhiwH-hatPHEkvH}
B(u_h; \Phi (w_H) -  \hat{P}_H E_k , v_H)   =  - R(u_h + t(w_H - E_k); u_h, u_h^k + w_H, v_H). 
\end{equation}
Substituting $v_H = \Phi (w_H) -  \hat{P}_H E_k$ into \eqref{Eqn:gHz}
and using \eqref{Eqn:BuhPhiwH-hatPHEkvH}, Remark \ref{Rem:Assaf}, H\"{o}lder inequality, \eqref{Eqn:Ek}, triangle inequality, \eqref{Eqn:ghz11}, \eqref{Eqn:hatPH1infty}, \eqref{Eqn:QH} and $\|E_k\|_{1, \infty} \lesssim \mid\log h\mid H^2$, it is obtained that
\begin{eqnarray*}
\partial (\Phi (w_H) -  \hat{P}_H E_k) (\boldsymbol{x})
&= & B(u_h; \Phi (w_H) -  \hat{P}_H E_k , g_H^{\boldsymbol{x}}) 
%
\\
 &=&  - R(u_h + t(w_H - E_k); u_h, u_h^k + w_H, g_H^{\boldsymbol{x}})
%
\\
&\lesssim& \|E_k - w_H\|_{1, \infty}^2 \| g_H^{\boldsymbol{x}}\|_{1, 1}
%
\\
&\lesssim&(  \|E_k - \hat{P}_H E_k \|_{1, \infty}^2 + \| \hat{P}_H E_k  - w_H\|_{1, \infty}^2 ) \mid\log H\mid  
\\
&\lesssim&  (  (1+\mid\log H \mid)^2  \|E_k\|_{1, \infty}^2 + H^2 )\mid\log H\mid 
%
\\
&\lesssim& (  (1+\mid\log H \mid)^2 \mid\log h\mid ^2 H^{4} + H^2 ) \mid\log H\mid. 
\end{eqnarray*}
Further using the arbitrariness of $\boldsymbol{x}$ and \eqref{Eqn:1infty}, the proof is finished.
\end{proof}

\begin{lemma}\label{Lem:continuous}
Assume $\|E_k\|_{1, \infty} \lesssim \mid\log h\mid H^2$, then the operator $\Phi$ is continuous in $V_H$. 
\end{lemma}

\begin{proof}
For any $w_1, w_2 \in Q_H$, by \eqref{Eqn:BPhiwH}, we have
\begin{eqnarray}\nonumber
B(u_h; \Phi (w_1) - \Phi (w_2) , v_H)  
 &=&     R(u_h + t(w_2 - E_k); u_h, u_h^k + w_2, v_H)
\\ \label{Eqn:BPhiw1-Phiw2}
 && - R(u_h + t(w_1 - E_k); u_h, u_h^k + w_1, v_H). 
\end{eqnarray}
Noticing that the definition of $R$ in \eqref{Eqn:A-A}, for the terms concerning $\boldsymbol{a}_{yy}$ on the right hand side of \eqref{Eqn:BPhiw1-Phiw2}, we can use Remark \ref{Rem:Assaf} and H\"{o}lder inequality to obtain that
\begin{eqnarray} \nonumber
\lefteqn{(\boldsymbol{a}_{yy}(u_h + t(w_2 - E_k))(E_k - w_2)^2, \nabla v_H) - (\boldsymbol{a}_{yy}(u_h + t(w_1 - E_k))(E_k - w_1)^2, \nabla v_H) }
\\ \nonumber
&=& (\boldsymbol{a}_{yy}(u_h + t(w_2 - E_k))(E_k - w_2)^2, \nabla v_H) 
 - (\boldsymbol{a}_{yy}(u_h + t(w_1 - E_k))(E_k - w_2)^2, \nabla v_H)
\\ \nonumber
&&+ (\boldsymbol{a}_{yy}(u_h + t(w_1 - E_k))(E_k - w_2)^2, \nabla v_H)
- (\boldsymbol{a}_{yy}(u_h + t(w_1 - E_k))(E_k - w_1)^2, \nabla v_H)
%
\\ \nonumber
&=& ( \left[\boldsymbol{a}_{yy}(u_h + t(w_2 - E_k))  - \boldsymbol{a}_{yy}(u_h + t(w_1 - E_k))  \right] (E_k - w_2)^2, \nabla v_H) 
\\   \nonumber
&& + (\boldsymbol{a}_{yy}(u_h + t(w_1 - E_k)) \left( -2 E_k w_2 + w_2^2 + 2E_k w_1 - w_1^2  \right), \nabla v_H)
%
\\  \nonumber
&=& 
( \left[\boldsymbol{a}_{yy}(u_h + t(w_2 - E_k))  - \boldsymbol{a}_{yy}(u_h + t(w_1 - E_k))  \right] (E_k - w_2)^2, \nabla v_H) 
\\   \nonumber
&&
 + (\boldsymbol{a}_{yy}(u_h + t(w_1 - E_k)) \left( 2E_k  - w_1 - w_2  \right)(w_1 - w_2), \nabla v_H)
\\ \nonumber
&\lesssim&  
\| \boldsymbol{a}_{yy}(u_h + t(w_2 - E_k))  - \boldsymbol{a}_{yy}(u_h + t(w_1 - E_k)) \|_{0, \infty} \|(E_k - w_2)^2\|_0 \| v_H\|_1
\\  \label{Eqn:ayy-ayy}
&&  
+ \|2E_k - w_1 - w_2  \|_0 \| w_1 - w_2\|_{0, \infty} \| v_H\|_1. 
\end{eqnarray}

For $\|(E_k - w_2)^2\|_0$, we use triangle inequality, \eqref{Eqn:hatPH1infty}, \eqref{Eqn:QH} and $\|E_k\|_{1, \infty} \lesssim \mid\log h\mid H^2$, it's obtained that 
\begin{eqnarray} \nonumber
\|(E_k - w_2)^2\|_0  
&\leq& \| E_k - w_2\|_{1, \infty}^2
\\ \nonumber
&\lesssim& \| E_k \|_{1, \infty}^2 + \| \hat{P}_H E_k\|_{1, \infty}^2 + \| \hat{P}_H E_k   - w_2\|_{1, \infty}^2
\\ \nonumber
&\lesssim& \| E_k \|_{1, \infty}^2  +  \mid\log H\mid ^2 \| E_k \|_{1, \infty}^2 + H^2
\\ \nonumber
&\lesssim& \mid\log h\mid ^2 H^4  + \mid\log H\mid ^2 \mid\log h\mid ^2 H^4  + H^2
\\ \label{Eqn:C1H}
& := & C_1(H), 
\end{eqnarray}
where $C_1(H)$ is a constant depending on $H$. 

Similarly, for $\|2E_k - w_1 - w_2  \|_0 $, there also exists a constant $C_2(H)$ such that
\begin{eqnarray}\label{Eqn:C2H}
\|2E_k - w_1 - w_2  \|_0  \lesssim C_2(H). 
\end{eqnarray}

Substituting \eqref{Eqn:C1H} and \eqref{Eqn:C2H} into \eqref{Eqn:ayy-ayy}, it could be obtained that
\begin{eqnarray*}
\lefteqn{(\boldsymbol{a}_{yy}(u_h + t(w_2 - E_k))(E_k - w_2)^2, \nabla v_H) - (\boldsymbol{a}_{yy}(u_h + t(w_1 - E_k))(E_k - w_1)^2, \nabla v_H) }
\\
&\lesssim& C(H)  \Big[ \| \boldsymbol{a}_{yy}(u_h + t(w_2 - E_k))  - \boldsymbol{a}_{yy}(u_h + t(w_1 - E_k)) \|_{0, \infty}
+  \| w_1 - w_2\|_{0, \infty}  \Big] \| v_H\|_1,
\end{eqnarray*}
where $C(H) = \max\{ C_1(H), C_2(H)\}$ . 
The rest of the items on the right hand side of \eqref{Eqn:BPhiw1-Phiw2} have similar results, and here is omitted. 
The conclusion follows from the above discussion, \eqref{Eqn:BPhiw1-Phiw2}, the BB-conditions of form $B(u_h; \cdot , \cdot)$ (See Remark \ref{Rem:infsupBuh}) and the continuity of second order derivatives  of $\boldsymbol{a}(\cdot, \cdot, \cdot)$ and $f(\cdot, \cdot, \cdot)$ (See the assumptions about $\boldsymbol{a}(\cdot, \cdot, \cdot)$ and $f(\cdot, \cdot, \cdot)$ in Section \ref{Sec:model}). 
\end{proof}

At last, we present the proof of Lemma \ref{Lem:eHk1infty} by Brouwer fixed point theorem.

\begin{proof}
Making use of Lemmas \ref{Lem:subset} and \ref{Lem:continuous} and Brouwer fixed point theorem, we know that \eqref{Eqn:BPhiwH} exists a fixed point $\bar{w}_H$ in $Q_H$.

Taking $w = u_h^k + \bar{w}_H$, $v = u_h$ and $\chi = v_H$ into \eqref{Eqn:A-A}, and then using \eqref{Eqn:uh} with $V_H \subset V_h$, Remark \ref{Rem:Bbilinear}, \eqref{Eqn:Ek}, $\bar{w}_H = \Phi(\bar{w}_H)$ and \eqref{Eqn:BPhiwH}, we obtain that 
\begin{eqnarray}\nonumber
A(u_h^k + \bar{w}_H, v_H) 
&=& A(u_h, v_H) + B(u_h; u_h^k + \bar{w}_H- u_h, v_H) + R(\tilde{\eta}; u_h, u_h^k + \bar{w}_H, v_H)
%
\\ \nonumber
&=& B(u_h; \bar{w}_H , v_H)  - B(u_h; E_k  , v_H)  + R(\tilde{\eta}; u_h, u_h^k + \bar{w}_H, v_H)
%
\\ \nonumber
&=& B(u_h;  \Phi(\bar{w}_H) , v_H)  - B(u_h; E_k  , v_H)  + R(\tilde{\eta}; u_h, u_h^k + \bar{w}_H, v_H)
%
\\ \label{Eqn:Auhk+wHvH}
&=& 0, 
\end{eqnarray}
where $\tilde{\eta} = u_h +  t ( \bar{w}_H - E_k)$. 
By the uniqueness of finite element solution (See Lemma \ref{Lem:uh}),  \eqref{Eqn:eHk} and \eqref{Eqn:Auhk+wHvH}, we can see that $\bar{w}_H = e_H^k$, which implies $e_H^k \in Q_H$.

At last, using triangle inequality, \eqref{Eqn:QH}, \eqref{Eqn:hatPH1infty} and $\| E_k\|_{1, \infty} \lesssim \mid\log h\mid H^{2}$, we obtain
\begin{eqnarray*}
\| e_H^k\|_{1, \infty} &\leq&    \| e_H^k - \hat{P}_H E_k\|_{1, \infty}   +   \|  \hat{P}_H E_k \|_{1, \infty} 
%
\\
&\lesssim&  H +  \mid\log H\mid   \| E_k\|_{1, \infty} 
%
\\
&\lesssim&  H +  \mid\log H\mid  \mid\log h\mid  H^2,  
\end{eqnarray*}
which completes the proof. 
\end{proof}


\begin{thebibliography}{10}
\bibitem{BiCJWangC18:23}
{\sc C.~J. Bi, C.~Wang, and Y.~P. Lin}, {\em A posteriori error estimates of
  two-grid finite element methods for nonlinear elliptic problems}, J. Sci.
  Comput., 74 (2018), 23--48.

\bibitem{Ciarlet02Book}
{\sc P.~G. Ciarlet}, {\em The Finite Element Method for Elliptic Problems},
  Classics in Applied Mathematics, No. 40. SIAM, Philadelphia, 2002.

\bibitem{GudiNataraj08:233}
{\sc T.~Gudi, N.~Nataraj, and A.~Pani}, {\em {$hp$}-discontinuous {Galerkin}
  methods for strongly nonlinear elliptic boundary value problems}, Numer.
  Math., 109 (2008), 233--268.

\bibitem{ThomeeXu89:553}
{\sc V.~Thom{\'e}e, J.~C. Xu, and N.~Y. Zhang}, {\em Superconvergence of the
  gradient in piecewise linear finite-element approximation to a parabolic
  problem}, SIAM J. Numer. Anal., 26 (1989), 553--573.

\bibitem{WeiHYHuangYQFEALPy}
{\sc H.~Y. Wei and Y.~Q. Huang}, {\em Fealpy: Finite element analysis library
  in python. https://github.com/weihuayi/ fealpy}, Xiangtan University,
  2017-2021.

\bibitem{XuJC92Meeting}
{\sc J.~C. Xu}, {\em Iterative methods by {SPD} and small subspace solvers for
  nonsymmetric or indefinite problems}, in Proceedings of the 5th International
  Symposium on Domain Decomposition Methods for Partial Differential Equations,
  Siam, Philadelphia, 1992, pp.~106--118.

\bibitem{XuJC94:231}
{\sc J.~C. Xu}, {\em A novel two-grid method for semilinear elliptic
  equations}, SIAM J. Sci. Comput., 15 (1994), 231--237.

\bibitem{XuJC94:79}
{\sc J.~C. Xu}, {\em Some two-grid finite element methods}, in Domain
  Decomposition Methods in Science and Engineering (Quarteroni, Alfio and
  P{\'e}riaux, Jacques and Kuznetsov, Yuri A and Widlund, Olof B eds), vol.~157
  of Contemp. Math., Amer. Math. Soc., Providence, RI, 1994, pp.~79--87.

\bibitem{XuJC96:1759}
{\sc J.~C. Xu}, {\em Two-grid discretization techniques for linear and
  nonlinear {PDEs}}, SIAM J. Numer. Anal., 33 (1996), 1759--1777.

\bibitem{XuJCCaiXC92:311}
{\sc J.~C. Xu and X.~C. Cai}, {\em A preconditioned {GMRES} method for
  nonsymmetric or indefinite problems}, Math. Comp., 59 (1992), 311--319.

\bibitem{ZhangWFFanRH20:522}
{\sc W.~F. Zhang, R.~H. Fan, and L.~Q. Zhong}, {\em Iterative two-grid methods
  for semilinear elliptic equations}, Comput. Math. Appl., 80 (2020),
   522--530.
\end{thebibliography}
\end{document}